\documentclass[11pt]{amsart} 
\usepackage[lmargin=1in,rmargin=1in,tmargin=1in,bmargin=1in]{geometry}
\usepackage[ps,all,arc,rotate]{xy}
\usepackage{graphicx, float, epstopdf}
\usepackage{bbm}
\usepackage{color}
\usepackage[unicode,bookmarks=false]{hyperref}
\hypersetup{hidelinks}
\usepackage{centernot}
\usepackage{fancyhdr}
\usepackage{multirow}
\usepackage[utf8]{inputenc}
\usepackage{amsfonts,amssymb,amsmath,amsthm,mathrsfs}
\usepackage{graphics, setspace}
\usepackage{braket}
\usepackage{mathtools}
\usepackage{tikz}  
\usepackage{upgreek}
\usepackage{xcolor}
\usepackage{array,esint}

\numberwithin{equation}{section}
\numberwithin{figure}{section}
\allowdisplaybreaks[4]         
 
\newtheorem{lemma}{Lemma}[section]
\newtheorem{theorem}{Theorem}[section]
\newtheorem{proposition}{Proposition}[section]

\newtheorem{corollary}[lemma]{Corollary}
\theoremstyle{definition}

\newtheorem{remark}{Remark}[section]
\usepackage{booktabs}
\usepackage{array}
\usepackage{vcell}
\newcommand{\Z}{\mathbb{Z}}

\newcommand{\R}{\mathbb{R}}
\newcommand{\C}{\mathbb{C}}
\newcommand{\N}{\mathbb{N}}

\newcommand{\e}{\operatorname{e}}

\usepackage{needspace}
\newcommand{\FA}{\mathcal F_A}
\newcommand{\ind}{\mathbf 1}

\newcommand{\eps}{\varepsilon}

\title[Exponential sums with multiplicative coefficients]{Effective estimates for exponential sums with multiplicative coefficients}
\author[N.~Robles]{Nicolas Robles}
\address{RAND Corporation, Engineering and Applied Sciences, Arlington, VA 22202, USA}
\email{robles.nicolas.m@gmail.com}

\begin{document}

\begin{abstract}
Let $f$ be multiplicative, with $|f(p)|\le A$ at primes and
$\sum_{n\le x}|f(n)|^2\le A^2x$ for every $x\ge1$.
If $|\alpha-a/q|\le q^{-2}$, $(a,q)=1$, and $3\le R\le q\le N/R$, we prove
\[
 \sum_{n\le N}f(n)\e(n\alpha)
 \ll_A \frac{N}{\log N}
       +\frac{N}{\sqrt R}\sqrt{\log\log(3R)}
\]
with effective implied constants. Montgomery and Vaughan
proved this with second term $NR^{-1/2}(\log R)^{3/2}$, and, for
$1$-bounded functions, Bachman replaced it by
$NR^{-1/2}\sqrt{\log R\log\log R}$. We remove the factor
$\sqrt{\log R}$ from Bachman's second term while retaining the original
coefficient hypotheses of Montgomery and Vaughan. A more precise estimate records the distance from a rational
number. The proof combines the Brun-Titchmarsh inequality
on short intervals with maximal Fourier estimates derived from the
Carleson-Hunt theorem; the local bounds permit arbitrary
prime-dependent prefixes. We also prove sharpness of the square-root
displacement dependence.
\end{abstract}

\maketitle

\section{Introduction}\label{sec:intro}

\subsection{Previous literature and new results}

Write $\N=\{1,2,\ldots\}$ and $\e(t)=\exp(2\pi i t)$. A multiplicative
function is a function $f:\N\to\C$ with $f(1)=1$ and $f(mn)=f(m)f(n)$
whenever $(m,n)=1$. For fixed $A\ge1$, let $\FA$ be the class of
multiplicative functions introduced by Montgomery and Vaughan
\cite[Section~1]{MV}, namely those satisfying
\begin{equation}\label{eq:class}
 |f(p)|\le A\quad(p\ \text{prime}),
 \qquad
 \sum_{n\le x}|f(n)|^2\le A^2x\quad(x\ge1).
\end{equation}
For $f\in\FA$, real $x\ge1$, and $\alpha\in\R$, put
\begin{equation}\label{eq:S-def}
 S_f(x,\alpha)=\sum_{n\le x}f(n)\e(n\alpha).
\end{equation}
The bounds below are uniform in $f$: the function may depend on both
the length and the phase.

An early result is Daboussi's theorem for multiplicative functions
with $|f(n)|\le1$: for each irrational $\alpha$, one has
$S_f(N,\alpha)=o(N)$ as $N\to\infty$; see
\cite[Theorem~1, p.~322]{D}, an account based on joint work with
Delange, and the note \cite{DD}, which its authors describe as the
starting point of \cite{MV} (\cite[p.~246]{DD2}). The detailed proofs
appeared in \cite{DD2}, whose Theorem~1 gives the same conclusion under
the mean-square condition $\sum_{n\le x}|f(n)|^2=O(x)$, without a
separate bound at primes, hence for every $f\in\FA$.
For a quantitative statement in $\FA$, we use the
formulation recorded and attributed to Daboussi in
\cite[equation~(3), p.~69]{MV}: if $|\alpha-a_0/q|\le q^{-2}$,
$(a_0,q)=1$, and $3\le q\le(N/\log N)^{1/2}$, then
$|S_f(N,\alpha)|\ll_A N(\log\log q)^{-1/2}$ uniformly for $f\in\FA$.
Montgomery and Vaughan
\cite[Corollary~1]{MV} sharpened it as follows: if
$|\alpha-a_0/q|\le q^{-2}$, $(a_0,q)=1$, and $2\le R\le q\le N/R$,
then
\begin{equation}\label{eq:old-mv-intro}
 |S_f(N,\alpha)|
 \ll_A \frac{N}{\log N}
       +\frac{N}{\sqrt R}(\log R)^{3/2}.
\end{equation}
For multiplicative functions satisfying $|f(n)|\le1$ for every $n$,
Bachman \cite[Theorem~5]{Bachman} replaced the second term by
\begin{equation}\label{eq:bachman-R}
 \frac{N}{\sqrt R}\sqrt{\log R\,\log\log R}\qquad(R\ge3).
\end{equation}
The term $N/\log N$ cannot be uniformly decreased in this class;
see \cite[Section~7]{MV}. The question of the optimal dependence on
$R$ is discussed in \cite[Section~1]{BBK}. Here the improvement is
relative to the specified estimates \eqref{eq:old-mv-intro} and
\eqref{eq:bachman-R}; no complete optimality statement for $R$ is made.
We remove $\sqrt{\log R}$ from the term \eqref{eq:bachman-R}, retain
the original class \eqref{eq:class}, and obtain effective constants.
The more precise statement records the displacement from a rational
number. Throughout, $\varphi$ denotes Euler's totient function, with
$\varphi(1)=1$.

\begin{theorem}\label{thm:main}
Let $A\ge1$, $f\in\FA$, and $N\ge3$. Put
\begin{equation}\label{eq:Q0}
 L=\log N,\qquad Q_0=N/L^3.
\end{equation}
Suppose that $a\in\Z$ and $r\in\N$ satisfy
\begin{equation}\label{eq:main-approx}
 (a,r)=1,\qquad 1\le r\le Q_0,\qquad
 \alpha=\frac ar+\beta,\qquad |\beta|\le\frac1{rQ_0}.
\end{equation}
With $B=\max\{1,N|\beta|\}$, one has
\begin{equation}\label{eq:main-bound}
 |S_f(N,\alpha)|
 \ll_A \frac{N}{\log N}
       +\frac{N}{\sqrt{\varphi(r)B}}.
\end{equation}
The bound is uniform in $f,N,a,r,\beta$ subject to these hypotheses,
and the implied constant is effective.
\end{theorem}

\begin{corollary}\label{cor:R}
Let $A\ge1$, $f\in\FA$, and $N\ge3$. Suppose that $a_0\in\Z$, $q\in\N$,
and $R\in\R$ satisfy
\begin{equation}\label{eq:R-approx}
 (a_0,q)=1,\qquad
 \bigg|\alpha-\frac{a_0}{q}\bigg|\le q^{-2},\qquad
 3\le R\le q\le N/R.
\end{equation}
Then, uniformly in these parameters,
\begin{equation}\label{eq:R-bound}
 |S_f(N,\alpha)|
 \ll_A \frac{N}{\log N}
       +\frac{N}{\sqrt R}\sqrt{\log\log(3R)}.
\end{equation}
The implied constant is effective.
\end{corollary}

Both results are unconditional. Effective means that a bound for the
constant can in principle be computed from $A$, not that a numerical
value is supplied; see Section~\ref{sec:effectivity}.

The hypotheses \eqref{eq:class} allow unbounded functions. For example,
put $w(n)=\sigma(n)/n=\sum_{d\mid n}d^{-1}$, where
$\sigma(n)=\sum_{d\mid n}d$. Factoring divisors of coprime products
shows that $w$ is multiplicative, and $w(p)=1+1/p\le3/2$. For every
real $x\ge1$,
\[
 \sum_{n\le x}w(n)^2
 =\sum_{d,e\ge1}\frac1{de}\Big\lfloor\frac x{[d,e]}\Big\rfloor
 \le x\sum_{g\ge1}\frac1{g^3}\sum_{a,b\ge1}\frac1{a^2b^2}
 \le\frac{245}{64}x<4x.
\]
Here $[d,e]$ denotes the least common multiple; the inequality follows
by writing $d=ga$, $e=gb$, $(a,b)=1$, and dropping the coprimality
condition. The bounds $\sum_{m\ge1}m^{-2}\le7/4$ and
$\sum_{m\ge1}m^{-3}\le5/4$ follow by comparing the tails from $m=3$
with the respective integrals over $[2,\infty)$. Thus
$w\in\mathcal F_2$, although
$w(m!)\ge\sum_{d\le m}d^{-1}\to\infty$. More generally, if
$f\in\FA$ and $u$ is a pointwise $1$-bounded multiplicative function,
then $fu\in\FA$ with the same $A$, since both inequalities in
\eqref{eq:class} are preserved. In particular, $wu\in\mathcal F_2$
uniformly in $u$. Applying \cite[Lemma~1]{DD2} to $g=|f|^2$ also gives, for every $M>1$,
\[
 \liminf_{x\to\infty}\frac{1}{\log\log x}
 \sum_{\substack{p\le x\\ |f(p)|^2<M}}\frac1p
 \ge 1-\frac1M.
\]
This example illustrates the coefficient hypotheses;
no comparison with specialized estimates for this weight is intended.

Related work extends the coefficients and describes the structure of
large sums. Jiang, L\"u, and Wang \cite[conditions (C.1)-(C.3) and
Theorem~1.1]{JLW} replace bounded prime values by averaged
prime-coefficient conditions and a weighted prime-pair sieve hypothesis,
with automorphic applications. Our global theorem retains
\eqref{eq:class}; the more general local estimates below do not by
themselves extend it to that framework.
De la Bret\`eche and Granville \cite[Theorem~4 and Section~10.1]{dBG}
describe large sums for $1$-bounded completely multiplicative functions
through character twists. Their refined conjecture
\cite[equation~(1.4)]{dBG} has $(r(1+N|\beta|))^{-1/2}$ in place of
$(\varphi(r)B)^{-1/2}$, and their main-term analysis already exhibits
square-root displacement decay.
Granville and Lamzouri \cite[Theorem~1.1, Corollary~1.1, Section~1.4]{GL}
work with $1$-bounded multiplicative functions and link large sums to
additive bias on large primes for $|\alpha-a/q|\le(qN)^{-1}$ and
$(\log N)^{2+\eps}\le q\le N/(\log N)^{3+\eps}$, $0<\eps<1/10$.
They summarize $N/\log N+N/\sqrt q$ as proved outside
$q=(\log N)^{2+o(1)}$, and there if $\varphi(q)\gg q$. Our effective
displacement result retains $\FA$; the optimal factor remains open.
Maier and Sankaranarayanan \cite[Theorem~1.2]{MS05} reprove, with
Siegel-Walfisz, the bound $x/\log x+x/\sqrt r$ for $|f|\le1$,
$\alpha=s/r$ with $r$ prime and $(s,r)=1$, and
$r\le x/((\log x)^2\log\log x\,\log\log\log x)$ for large $x$.
This case of Bachman's Theorems~4 and~5 follows from
Theorem~\ref{thm:main} and Corollary~\ref{cor:R}.

\subsection{Comparison and the logarithmic transition}

Under \eqref{eq:main-approx}, Bachman's Theorem~4 gives
$|S_f(N,\alpha)|\ll N/L+N/\sqrt{\varphi(r)}$ for $1$-bounded
multiplicative functions, and the estimate (3.1) in his proof records
the displacement dependence
\begin{equation}\label{eq:bachman-B}
 |S_f(N,\alpha)|\ll \frac{N}{L}
   +\frac{N}{\sqrt{\varphi(r)B}}\sqrt{\log(2B)};
\end{equation}
see \cite[Theorem~4 and equation~(3.1)]{Bachman}. The gain in
Theorem~\ref{thm:main} is the removal of $\sqrt{\log(2B)}$, not the
case $B=1$. The common term $N/L$ means that the whole estimate need
not improve by an unbounded factor in every range;
Remark~\ref{rem:comparison} below gives a rational family on which it
does, for Theorem~\ref{thm:main} against \eqref{eq:bachman-B} and for
Corollary~\ref{cor:R} against \eqref{eq:old-mv-intro} and
\eqref{eq:bachman-R}.

Writing $L=\log N$, the second term in \eqref{eq:old-mv-intro} is
$O(N/L)$ when $R$ has order $L^2(\log L)^3$. Bachman's refinement
\eqref{eq:bachman-R} reaches that scale when $R$ has order
$L^2\log L\,\log\log L$, whereas Corollary~\ref{cor:R} requires only
$R$ of order $L^2\log\log L$. At each scale,
$\log R\asymp\log L$ and $\log\log R\asymp\log\log L$, so the claims
follow by substitution as $N\to\infty$. These are sufficient scales
for the displayed bounds, not necessary conditions for every method
or choice of rational approximation.

\begin{remark}\label{rem:comparison}
The following family shows that the whole bound can improve by an
unbounded factor. The comparison is restricted to $1$-bounded
multiplicative functions,
the class common to the bounds discussed here. Let $r=\prod_{p\le z}p$
run through the primorials, and put
\[
 H=\frac r{\varphi(r)},\qquad N=\lfloor\exp(\sqrt{rH})\rfloor,\qquad
 L=\log N,
\]
so that $L^2\sim rH$. The Euler product for $H$ contains the reciprocal
of every integer at most $z$, so $H\ge\sum_{n\le z}1/n\to\infty$, while
$H\ll\log\log(3r)$ by \eqref{eq:totient}. Let $M$ be the least integer
with $M\equiv-1\pmod r$ and $M\ge N/(rH^2)$; its difference from
$N/(rH^2)$ is less than $r$, and $r^2H^2/N\to0$, so $M\sim N/(rH^2)$.
Put
\[
 \alpha=\frac{(M+1)/r}M=\frac1r+\frac1{rM}.
\]
For the approximation $a/r=1/r$ one has $\beta=1/(rM)$,
$B=N/(rM)\sim H^2>1$, and
\[
 \varphi(r)B=\frac N{HM}\sim rH\sim L^2.
\]
Moreover $r/Q_0\to0$ and $Q_0/M\sim rH^2/L^3\sim H/L\to0$, so
$r\le Q_0$, $M\ge Q_0$, and $|\beta|\le1/(rQ_0)$: the hypotheses
\eqref{eq:main-approx} hold, and Theorem~\ref{thm:main} gives
$|S_f(N,\alpha)|\ll_A N/L$. The right-hand side of
\eqref{eq:bachman-B} at the same approximant is
$\asymp(N/L)\sqrt{\log(2H)}$, which exceeds $N/L$ by an unbounded
factor, and $1/r$ is the only admissible approximant: if a reduced
fraction $b/s\ne1/r$ satisfied $s\le Q_0$ and $|\alpha-b/s|\le1/(sQ_0)$,
then
\[
 \frac1{rs}\le\bigg|\frac bs-\frac1r\bigg|\le\frac1{sQ_0}+\frac1{rM},
 \qquad\text{so that}\qquad
 1\le\frac r{Q_0}+\frac sM\le\frac r{Q_0}+\frac{Q_0}M=o(1),
\]
a contradiction.

The same family separates the total bounds in the form of
Corollary~\ref{cor:R}. The fraction $\alpha=((M+1)/r)/M$ is reduced,
since a common divisor of $(M+1)/r$ and $M$ divides both $M$ and
$M+1$. Let $b/q$ be a reduced fraction with $|\alpha-b/q|\le q^{-2}$.
If $b/q\ne\alpha$, then $1/(Mq)\le|\alpha-b/q|\le q^{-2}$ gives
$q\le M$; if moreover $q>2r$, then $b/q\ne1/r$ and
\[
 \frac1{rq}\le\bigg|\frac bq-\frac1r\bigg|\le\frac1{q^2}+\frac1{rM},
 \qquad\text{so that}\qquad 1\le\frac rq+\frac qM,
\]
whence $q>M/2$. Thus every admissible denominator satisfies $q\le2r$
or $M/2<q\le M$. We also need $H\asymp\log\log r$: the upper bound is
\eqref{eq:totient}, and $H\ge\sum_{n\le z}1/n\ge\log z$ together with
$\log r=\sum_{p\le z}\log p\le z\log z$ gives
$\log\log r\le2\log z\le2H$ for $z\ge3$.

Corollary~\ref{cor:R} applies with $a_0/q=\alpha$, $q=M$, and
$R=N/M$, since $M^2\ge N$ for large $r$. As
$\sqrt{NM}\sim N/(H\sqrt r)\sim N/(L\sqrt H)$ and
$\log\log(3N/M)\le\log(2\log r)+o(1)\ll H$, its second term satisfies
\[
 \sqrt{NM}\sqrt{\log\log(3N/M)}\ll\frac NL,
\]
so that \eqref{eq:R-bound} gives $|S_f(N,\alpha)|\ll_AN/L$. In
\eqref{eq:old-mv-intro} and \eqref{eq:bachman-R}, on the other hand,
let $(a_0/q,R)$ be any admissible choice, so that $R\le2r$ if
$q\le2r$, and $R\le N/q<2N/M$ if $M/2<q\le M$. The functions $(\log t)^3/t$ and $\log t\log\log t/t$ decrease for
$t\ge21$. For $2\le R<21$ in \eqref{eq:old-mv-intro}, and for
$3\le R<21$ in \eqref{eq:bachman-R}, the respective second term is
bounded below by a positive absolute multiple of $N$.
For $R\ge21$, the denominator dichotomy and monotonicity show that
Bachman's second term is bounded below by a positive absolute
multiple of
\[
 \begin{cases}
 (N/L)\sqrt{H\log r\log\log r},&q\le2r,\\
 (N/L)\sqrt{\log r\log\log r/H},&M/2<q\le M.
 \end{cases}
\]
Indeed these are its orders at $R=2r$ and $R=2N/M$, respectively.
Both are $\gg(N/L)\sqrt{\log r}$ because $H\asymp\log\log r$.
The same comparison for \eqref{eq:old-mv-intro} gives a second term
$\gg(N/L)(\log r)^{3/2}/\sqrt H$ for every admissible choice.
Consequently, when each displayed right-hand side is regarded as a
positive comparison expression with its implicit constant omitted,
its infimum over the corresponding admissible $(a_0/q,R)$ has order
\[
 \frac NL\frac{(\log r)^{3/2}}{\sqrt H},\qquad
 \frac NL\sqrt{\log r},\qquad \frac NL,
\]
for \eqref{eq:old-mv-intro}, \eqref{eq:bachman-R}, and
\eqref{eq:R-bound}, respectively. These orders are achieved, up to
absolute constant factors, by $q=M$ and $R=N/M$.
For this denominator, $N/q\sim rH^2\asymp L^2\log\log L$.
This compares direct substitutions into the specified expressions at
the same length $N$. It is not a lower bound for the actual sum, nor
a claim that every other estimate or indirect combination of methods
has the same limitation.
\end{remark}

\subsection{Overview of the proof}

A logarithmic identity reduces the problem to prime-integer bilinear
sums, and the original Montgomery-Vaughan bound removes the small
primes. In a remaining block $P<p\le2P$, Cauchy-Schwarz leaves a
nonnegative prime-weighted square. Brun-Titchmarsh bounds the prime
mass in reduced residue classes on intervals of length comparable
to $P/B$. An elementary variation inequality transfers this bound
to continuous maximal Fourier expressions. Carleson-Hunt and a
short-sequence estimate over reduced residues control those expressions
without a logarithmic loss. Propositions~\ref{prop:effective-maximal-prime}
and~\ref{prop:arbitrary-prefixes} allow arbitrary coefficient vectors
and prime-dependent prefixes; the global application uses only
$n\le N/p$. No stronger prime-pair sieve is required.
Section~\ref{sec:mechanism} compares these choices with the arguments
of Montgomery-Vaughan and Bachman and identifies the logarithmic
sums that occur in the quoted estimates.

Bachman also applies Cauchy-Schwarz in the prime variable. His
completed weight is treated by a Siegel-Walfisz asymptotic, introducing
an ineffective constant \cite[equations~(2.6)-(2.7)]{Bachman}.
His Theorems~4 and~5 assume $|f|\le1$; boundedness enters both
Lemma~1 and the passage to (3.1), for example in
\cite[equations~(2.6), (2.14), and~(3.1)]{Bachman}.
Our prime-block argument instead bounds a nonnegative expression
using Brun-Titchmarsh, without a progression asymptotic; see
Section~\ref{sec:mechanism}.

The term $N/L$ cannot be uniformly decreased throughout $\FA$:
Montgomery and Vaughan \cite[Section~7, example~(i)]{MV} construct,
for each prescribed $N,\alpha$, a completely multiplicative
$1$-bounded function with $|S_f(N,\alpha)|\gg N/L$.
Proposition~\ref{prop:sharpness} shows that $B^{-1/2}$ is also sharp
in the stated uniform setting. Neither example settles the optimal
totient dependence for multiplicative coefficients; see
Remark~\ref{rem:totient-obstruction}.
The displacement dependence in \eqref{eq:main-bound} also yields, in
Corollary~\ref{cor:simultaneous}, an effective bound
$\ll_A(\log N)^4/N$ for the measure of a single exceptional set outside
which $|S_f(N,\alpha)|\le c_AN/\log N$ holds simultaneously for every
$f\in\FA$.

\subsection{Notation}\label{sec:notation}

Sums with real upper limits run over positive integers up to their
floor; upper limits below $1$ give empty sums. A dyadic number is a
power of two, $(m,n)$ is the greatest common divisor, and
$\ind_{\mathcal E}$ is $1$ when the indicated condition holds and $0$
otherwise. We write $X\ll_A Y$, for $Y\ge0$, when $|X|\le C(A)Y$ with
a constant depending only on $A$; an omitted subscript denotes an
absolute constant. For positive quantities, $X\asymp Y$ means
$X\ll Y\ll X$, and $X\sim Y$ means $X/Y\to1$ in the limit considered.
The notation $O_A(Y)$ has the same constant dependence as $\ll_A Y$.

\subsection{Organization of the paper}

Sections~\ref{sec:inputs}-\ref{sec:fourier} collect the imported
estimates and the new lemmas. Section~\ref{sec:mechanism} compares
the earlier arguments with the choices used here.
Section~\ref{sec:effective-sieve} contains the prime-block estimates,
Section~\ref{sec:proof} the proofs of the main results and the
exceptional-set corollary, and Section~\ref{sec:sharpness} treats
sharpness before the conclusion and the directions for future work.

\section{Preliminaries and imported estimates}\label{sec:inputs}

We first record the external inputs used in the main proof.
Their constants are effective. The new estimates
\eqref{eq:main-bound} and \eqref{eq:R-bound} are not used in any of the
reductions in this section or in Section~\ref{sec:small}.

\begin{lemma}\label{lem:elementary}
For $f\in\FA$ and $x\ge1$,
\begin{equation}\label{eq:l1-l2}
 \sum_{n\le x}|f(n)|\le Ax,
 \qquad
 \sum_{n\ge1}\frac{|f(n)|^2}{n^{4/3}}\le4A^2.
\end{equation}
For $x\ge2$,
\begin{equation}\label{eq:prime-elementary}
 \vartheta(x):=\sum_{p\le x}\log p\ll x,
 \qquad
 \sum_{p\le x}\frac{\log p}{p}\ll\log(2x).
\end{equation}
For every integer $m\ge1$,
\begin{equation}\label{eq:totient}
 \frac m{\varphi(m)}\ll\log\log(3m).
\end{equation}
\end{lemma}

\begin{proof}
Cauchy's inequality and \eqref{eq:class} give
$\sum_{n\le x}|f(n)|\le \lfloor x\rfloor^{1/2}(A^2x)^{1/2}\le Ax$.
If $M_2(t)=\sum_{n\le t}|f(n)|^2$, partial summation gives
\[
 \sum_{n\ge1}\frac{|f(n)|^2}{n^{4/3}}
 =\frac43\int_1^\infty \frac{M_2(t)}{t^{7/3}}\,dt
 \le\frac43 A^2\int_1^\infty t^{-4/3}\,dt=4A^2.
\]
For an integer $m\ge1$, the product of primes in $(m,2m]$ divides
$\binom{2m}{m}$, so
$\vartheta(2m)-\vartheta(m)\le2m\log2$.
Summing this inequality at powers of two proves $\vartheta(x)\ll x$.
The identity
\[
 \sum_{p\le x}\frac{\log p}{p}
 =\frac{\vartheta(x)}x+\int_2^x\frac{\vartheta(t)}{t^2}\,dt
\]
gives the other estimate in \eqref{eq:prime-elementary}.
Finally, \eqref{eq:totient} is the usual uniform form of the lower bound
for $\varphi(m)$; see \cite[Corollary~3.6]{Kouk}. The finitely many small
values of $m$, including $m=1$, are absorbed in an absolute constant.
\end{proof}

\begin{proposition}[Montgomery-Vaughan]\label{prop:MV}
Let $f\in\FA$ and $x\ge3$ be real. If $u\in\Z$, $v\in\N$, and $R_0\ge2$
satisfy
\[
 (u,v)=1,\qquad |\gamma-u/v|\le v^{-2},\qquad
 R_0\le v\le x/R_0,
\]
then
\begin{equation}\label{eq:MV-input}
 |S_f(x,\gamma)|\ll_A
 \frac{x}{\log(2x)}+
 \frac{x}{\sqrt{R_0}}(\log(2R_0))^{3/2}.
\end{equation}
\end{proposition}

This is Corollary~1 of \cite{MV}; the proof below only records the
passage from integral to real $x$.

\begin{proof}
For integer $x$, this is \cite[Corollary~1, p.~70]{MV}, with harmless
changes from $\log x$ and $\log R_0$ to the displayed logarithms.
For real $x$, bounded $x$ or $2\le R_0<4$ follows from
$|S_f(x,\gamma)|\le Ax$.
Otherwise set $n=\lfloor x\rfloor\ge x/2$ and apply the integer result
with range parameter $R_0/2$. Indeed,
$R_0/2\le v\le x/R_0\le2n/R_0$, and $R_0/2\ge2$.
The resulting bound implies \eqref{eq:MV-input}, since $n\asymp x$.
The constant is effective: the proof of this corollary in
\cite[Sections~2-6]{MV} uses elementary estimates and effective
upper-bound sieve inequalities, not Siegel-Walfisz.
\end{proof}

\begin{lemma}\label{lem:dirichlet}
For every $Q\ge1$ and $\gamma\in\R$ there are $u\in\Z$ and $v\in\N$ with
\[
 (u,v)=1,\qquad 1\le v\le Q,\qquad
 |\gamma-u/v|\le\frac1{vQ}.
\]
\end{lemma}

\begin{proof}
Write $\{t\}=t-\lfloor t\rfloor$ for the fractional part of a real
number $t$. Put $m=\lfloor Q\rfloor$ and consider the $m+1$ points
$0,\{\gamma\},\ldots,\{m\gamma\}$ on $\R/\Z$.
If two coincide, a difference of their indices yields an exact rational
representation with denominator at most $m$.
Otherwise one circular gap has length at most $1/(m+1)$.
The difference of the corresponding indices gives $1\le v_0\le m$ and
an integer $u_0$ such that
$|v_0\gamma-u_0|\le1/(m+1)<1/Q$.
Reduce $u_0/v_0$ to $u/v$. Since $v\le v_0$, its approximation error is
at most $1/(v_0Q)\le1/(vQ)$.
\end{proof}

\begin{proposition}[Brun-Titchmarsh on an arbitrary interval]\label{prop:BT}
There is an effective absolute constant $C_\mathrm{BT}$ such that, for
real $x\ge h\ge r\ge1$, integers $r,b$ with $(b,r)=1$,
\begin{equation}\label{eq:BT}
 \#\{p:x-h<p\le x,\ p\equiv b\pmod r\}
 \le C_\mathrm{BT}\frac{h}{\varphi(r)\log(2h/r)}.
\end{equation}
The constant is independent of the location $x$ of the interval.
\end{proposition}

This is the Brun-Titchmarsh inequality in the form given by
Montgomery and Vaughan \cite{MV2}; the proof records the source used and
the endpoint convention.

\begin{proof}
This is the constant-factor consequence of
\cite[Theorem~21.4, p.~219]{Kouk}, whose proof is by Selberg's upper-bound
sieve. Writing $t=h/r$, its displayed numerator is
$2+O(\log\log(3t)/\log t)$, which is bounded effectively for $t\ge2$. For $1\le h/r<2$, counting
all integers in the residue class gives at most $h/r+1\le2h/r$,
which also implies \eqref{eq:BT}, since $\varphi(r)\le r$ and
$\log(2h/r)\le\log4$. This includes $r=1$.
\end{proof}

\begin{proposition}[Finite-prefix Carleson-Hunt inequality]\label{prop:CH}
There is an effective absolute constant $C_\mathrm{CH}$ such that for every integer
$U\ge1$ and all $c_1,\ldots,c_U\in\C$,
\begin{equation}\label{eq:CH}
 \int_0^1\max_{0\le K\le U}
 \bigg|\sum_{n\le K}c_n\e(n\theta)\bigg|^2\,d\theta
 \le C_\mathrm{CH}\sum_{n\le U}|c_n|^2.
\end{equation}
The maximum is over integer $K$, and the sum for $K=0$ is zero.
\end{proposition}

This is the $L^2$ case of the Carleson-Hunt theorem \cite{C,H} for
trigonometric polynomials; the proof records the passage from the
tail-maximal form in the source to the prefix-maximal form used here.

\begin{proof}
The classical Carleson-Hunt theorem is stated in tail-maximal form in
\cite[Theorem~1.2, p.~87]{HS}. Let $C_\mathrm{tail}$ denote the absolute
constant in that tail-maximal inequality. Apply it to the analytic
polynomial $F(z)=\sum_{n=1}^U c_nz^n$, with all other Taylor
coefficients zero. Its tail beginning at $K+1$ is
$F(z)-\sum_{n\le K}c_nz^n$.
The inequality $|u-v|^2\le2|u|^2+2|v|^2$, followed by the tail-maximal
bound and Parseval's identity, gives \eqref{eq:CH} with
$C_\mathrm{CH}=2(1+C_\mathrm{tail})$.
Normalized measure on the unit circle becomes $d\theta$ under
$z=\exp(2\pi i\theta)$. The quantitative proofs of the
Carleson-Hunt bound permit an effective absolute constant; see also
\cite{Lacey} for a proof and the equivalent
Fourier-series formulation. No numerical value of $C_\mathrm{CH}$ is
required here.
\end{proof}

\section{The logarithmic identity and small primes}\label{sec:small}

\begin{lemma}\label{lem:log}
Uniformly for $f\in\FA$, real $N\ge3$, and $\alpha\in\R$,
\begin{equation}\label{eq:log-identity}
 (\log N)S_f(N,\alpha)
 =\sum_{p\le N}f(p)S_f(N/p,p\alpha)\log p+O_A(N).
\end{equation}
\end{lemma}

\begin{proof}
By \eqref{eq:l1-l2} and partial summation,
\[
 \sum_{n\le N}|f(n)|\log(N/n)
 =\int_1^N\frac1t\sum_{n\le t}|f(n)|\,dt
 \le A(N-1).
\]
It therefore suffices to compare the sum weighted by $\log n$ with the
prime sum in \eqref{eq:log-identity}.
For $p^j\Vert n$ (meaning $p^j\mid n$ and $p^{j+1}\nmid n$),
write $n=p^jm$ with $p\nmid m$.
Using $\log n=\sum_{p^j\Vert n}j\log p$, multiplicativity shows that the
difference of the two coefficients belonging to this prime is
\[
 \bigl(jf(p^j)-f(p)f(p^{j-1})\bigr)f(m)\log p.
\]
For $j=1$ this is zero. The remaining error is bounded by
\begin{equation}\label{eq:prime-power-error}
 AN\sum_p\sum_{j\ge2}\frac{\log p}{p^j}
       \bigl(j|f(p^j)|+A|f(p^{j-1})|\bigr).
\end{equation}
The series with the first term satisfies
\[
 \sum_p\sum_{j\ge2}\frac{j|f(p^j)|\log p}{p^j}
 \le
 \bigg(\sum_{n\ge1}\frac{|f(n)|^2}{n^{4/3}}\bigg)^{1/2}
 \bigg(\sum_p\sum_{j\ge2}
       \frac{j^2(\log p)^2}{p^{2j/3}}\bigg)^{1/2}\ll_A1.
\]
For the convergence of the second squared factor, the sum over $j\ge2$
is $\ll p^{-4/3}(\log p)^2$; its sum over primes is bounded by the
convergent sum over integers. Distinct pairs $(p,j)$ give distinct prime
powers, which justifies the first factor.
The other series in \eqref{eq:prime-power-error} is at most
\[
 A^2\sum_p\frac{\log p}{p^2}
 +\frac A2\sum_p\sum_{j\ge2}
             \frac{|f(p^j)|\log p}{p^j}\ll_A1.
\]
Here the first term is the contribution from the exponent $2$ in
\eqref{eq:prime-power-error}; for the other exponents put $j$ equal to
that exponent minus one and use $p^{-1}\le1/2$.
This proves \eqref{eq:log-identity}. The reduction is the
phase-independent form of the preliminary reduction in \cite[Section~2]{MV};
the details above show explicitly that no bound on $f(p^j)$ for $j\ge2$
has been inserted.
\end{proof}

\begin{lemma}\label{lem:coarse}
Let $f\in\FA$, $\gamma=b/s+\delta$, $(b,s)=1$, and set
\[
 H=\max\{1,x|\delta|\},\qquad D_0=sH.
\]
If $x\ge16D_0^2$, then
\begin{equation}\label{eq:coarse}
 |S_f(x,\gamma)|\ll_A
 \frac{x}{\log(2x)}+
 \frac{x(\log(2D_0))^{3/2}}{\sqrt{D_0}}.
\end{equation}
\end{lemma}

\begin{proof}
For $1\le D_0<16$, the second term in \eqref{eq:coarse} is bounded
below by a positive absolute multiple of $x$, so \eqref{eq:l1-l2}
suffices. Assume $D_0\ge16$.
If $H=1$, then $D_0=s$, $|\delta|\le x^{-1}\le s^{-2}$, and
$s\le x/s$. Apply Proposition~\ref{prop:MV} with $R_0=s$.

Suppose $H>1$. Put $Q=\lfloor2x/D_0\rfloor$ and use
Lemma~\ref{lem:dirichlet} to choose a reduced fraction $c/t$ with
\[
 t\le Q,\qquad |\gamma-c/t|\le1/(tQ).
\]
Since $x/D_0\ge16D_0$, we have $Q\ge x/D_0$ and
\[
 \frac{D_0Q}{x}\ge2-\frac{D_0}{x}>1.
\]
Thus $b/s\ne c/t$: equality of these reduced fractions would give
$sQ|\delta|\le1$, whereas $sQ|\delta|=D_0Q/x>1$.
Separation of distinct reduced fractions yields
\[
 1\le st|\delta|+s/Q=D_0t/x+s/Q.
\]
The bound $s/Q\le sD_0/x\le D_0^2/x\le1/16$ therefore implies
\[
 \frac{x}{2D_0}\le t\le\frac{2x}{D_0}.
\]
In particular, for $R_0=D_0/4$, we have $R_0\ge2$,
$R_0\le t\le x/R_0$, and $|\gamma-c/t|\le t^{-2}$.
Applying Proposition~\ref{prop:MV} with these parameters gives \eqref{eq:coarse}.
\end{proof}

\begin{lemma}\label{lem:small}
Let $f\in\FA$, $N\ge3$, $L=\log N$, and
$\alpha=a/r+\beta$, where $a\in\Z$, $r\in\N$, and $(a,r)=1$.
Put $B=\max\{1,N|\beta|\}$ and $D=rB$.
If $2\le Y\le\sqrt N$ and $N/Y\ge16D^2$, then
\begin{equation}\label{eq:small-primes-general}
 \bigg|\sum_{p\le Y}f(p)S_f(N/p,p\alpha)\log p\bigg|
 \ll_A N\bigg(\frac{\log(2Y)}L
 +\frac{(\log(2D))^{3/2}\log(2Y)}{\sqrt D}
 +\frac{(\log(2D))^{5/2}}{\sqrt D}\bigg).
\end{equation}
In particular, suppose $D\le L^3$ and $Y$ is the least power of two
not smaller than $16D^2$. For all sufficiently large $N$, with an
effective absolute threshold,
\begin{equation}\label{eq:small-primes}
 \sum_{p\le Y}f(p)S_f(N/p,p\alpha)\log p=O_A(N).
\end{equation}
\end{lemma}

\begin{proof}
For $p\le Y$, put $x=N/p$ and $d_p=(p,r)$, so $d_p\in\{1,p\}$.
The reduced denominator of the rational part of $p\alpha$ is $r/d_p$,
while
\[
 \max\{1,x|p\beta|\}=B,\qquad D_0=(r/d_p)B=D/d_p.
\]
Thus $x\ge16D_0^2$, and Lemma~\ref{lem:coarse} applies. Also,
$\log(2x)\ge L/2$ because $p\le\sqrt N$.
Using $|f(p)|\le A$, the first term of \eqref{eq:coarse} contributes
\[
 \ll_A\frac NL\sum_{p\le Y}\frac{\log p}{p}
 \ll_A\frac{N\log(2Y)}L.
\]
For $p\nmid r$, its second term contributes
\[
 \ll_A\frac{N(\log(2D))^{3/2}}{\sqrt D}
                \sum_{p\le Y}\frac{\log p}{p}
 \ll_A\frac{N(\log(2D))^{3/2}\log(2Y)}{\sqrt D}.
\]
For $p\mid r$, use $D_0=D/p\ge1$ to bound that contribution by
\[
 \ll_A\frac{N(\log(2D))^{3/2}}{\sqrt D}
                \sum_{p\mid r}\frac{\log p}{\sqrt p}
 \ll_A\frac{N(\log(2D))^{5/2}}{\sqrt D}.
\]
Indeed the prime sum is empty for $r=1$, and otherwise it is at most
$\log r\le\log D$. This proves \eqref{eq:small-primes-general}.

For the stated dyadic choice,
\begin{equation}\label{eq:polynomial-cutoff}
 16D^2\le Y<32D^2,\qquad \log(2Y)\ll\log(2D).
\end{equation}
The conditions $Y\le\sqrt N$ and $N/Y\ge16D^2$ follow, uniformly for
$D\le L^3$, from $32L^6<\sqrt N$ and $512L^{12}<N$.
These inequalities hold beyond a computable absolute threshold.
The bound \eqref{eq:small-primes-general} is then
\[
 \ll_A N\Big(\frac{\log(2D)}L
 +\frac{(\log(2D))^{5/2}}{\sqrt D}\Big)\ll_A N.
\]
Here $\log(2D)\ll\log(2L)$, and
$(\log(2D))^{5/2}/\sqrt D$ is bounded for $D\ge1$.
All constants and thresholds used in this proof are effective.
\end{proof}

\section{Finite maximal Fourier estimates}\label{sec:fourier}

For $r\in\N$ write
\[
 \mathcal A(r)=\{0,1,\ldots,r-1\},\qquad
 \mathcal R(r)=\{b\in\mathcal A(r):(b,r)=1\}.
\]
In particular, $\mathcal R(1)=\{0\}$.
For an integer $U\ge1$ and a vector $c=(c_1,\ldots,c_U)\in\C^U$, set
\begin{equation}\label{eq:maximal-definition}
 E(c)=\sum_{n\le U}|c_n|^2,\qquad
 M_c(\theta)=\max_{0\le K\le U}
                  \bigg|\sum_{n\le K}c_n\e(n\theta)\bigg|.
\end{equation}
The maximum is over integer $K$, including the zero prefix.
The function $M_c$ is continuous and $1$-periodic.
Proposition~\ref{prop:CH} says that
$\int_0^1 M_c(\theta)^2\,d\theta\le C_\mathrm{CH}E(c)$.

\begin{lemma}\label{lem:short-residue}
For integers $1\le U\le r$ and any $c_1,\ldots,c_U\in\C$,
\begin{equation}\label{eq:short-residue}
 \frac1{\varphi(r)}\sum_{b\in\mathcal R(r)}
       \bigg|\sum_{n\le U}c_n\e(bn/r)\bigg|^2
 \ll \log\log(3U)\sum_{n\le U}|c_n|^2.
\end{equation}
The implied constant is absolute.
\end{lemma}

\begin{proof}
The case $U=1$ is immediate, since the left side equals $|c_1|^2$
and $\log\log3>0$.
For $U\ge2$, choose a divisor $d\mid r$ with $(d,r/d)=1$ and $d\ge U$,
and write $r=ds$.
The map $(u,v)\mapsto us+vd\pmod r$ is a bijection from
$\mathcal R(d)\times\mathcal R(s)$ to $\mathcal R(r)$.
The expression on the left of \eqref{eq:short-residue} therefore equals
\[
 \frac1{\varphi(s)}\sum_{v\in\mathcal R(s)}
 \frac1{\varphi(d)}\sum_{u\in\mathcal R(d)}
 \bigg|\sum_{n\le U}c_n\e(vn/s)\e(un/d)\bigg|^2.
\]
Enlarge the inner sum to all $u\bmod d$.
Since $n,m\le U\le d$, the condition $d\mid n-m$ forces $n=m$.
Orthogonality bounds the preceding expression by
\begin{equation}\label{eq:unitary-bound}
 \frac d{\varphi(d)}\sum_{n\le U}|c_n|^2.
\end{equation}

It remains to choose $d$. If a full prime-power factor $p^\nu\Vert r$
satisfies $p^\nu\ge U$, take $d=p^\nu$; then
$d/\varphi(d)=p/(p-1)\le2$.
Otherwise all full prime-power factors of $r$ are less than $U$.
Multiply them, in any order, until their product first reaches $U$.
Such a point is reached because their total product is $r\ge U$.
The selected product satisfies
\[
 U\le d<U^2,\qquad (d,r/d)=1.
\]
In this case \eqref{eq:totient} gives
$d/\varphi(d)\ll\log\log(3d)\ll\log\log(3U)$.
Insert either choice into \eqref{eq:unitary-bound}.
\end{proof}

\begin{lemma}\label{lem:maximal-sampling}
For integers $U,r\ge1$, $c\in\C^U$, and real $\theta$,
\begin{equation}\label{eq:maximal-all-residues}
 \sum_{b\bmod r}M_c(\theta+b/r)^2\ll (U+r)E(c).
\end{equation}
If $U\le r$, then also
\begin{equation}\label{eq:maximal-reduced-residues}
 \frac1{\varphi(r)}\sum_{b\in\mathcal R(r)}M_c(\theta+b/r)^2
 \ll\log\log(3U)E(c).
\end{equation}
Both constants are absolute and effective, and both bounds are uniform
in $\theta$.
\end{lemma}

\begin{proof}
For $1$-periodic functions use the convolution
$(g*h)(\theta)=\int_0^1g(\theta-v)h(v)\,dv$ and the Fourier coefficient
$\widehat g(k)=\int_0^1g(v)\e(-kv)\,dv$.
For an integer $m\ge1$, define the Fej\'er kernel
\[
 K_m(v)=\frac1m\bigg|\sum_{j=0}^{m-1}\e(jv)\bigg|^2
       =\sum_{|k|<m}\Big(1-\frac{|k|}m\Big)\e(kv).
\]
It is nonnegative and has integral $1$. Put
\[
 V_U=2K_{2U}-K_U,\qquad W_U=2K_{2U}+K_U.
\]
Then $\widehat V_U(k)=1$ for every integer $|k|\le U$,
$|V_U|\le W_U$, and $\int_0^1W_U=3$.
Consequently every prefix polynomial
$F_K(\theta)=\sum_{n\le K}c_n\e(n\theta)$ satisfies $F_K=F_K*V_U$.
Weighted Cauchy-Schwarz gives
\[
 |F_K(\theta)|^2
 \le 3\int_0^1 |F_K(\theta-v)|^2 W_U(v)\,dv.
\]
Taking the maximum only after this inequality, we obtain
\begin{equation}\label{eq:positive-reproducing-majorant}
 M_c(\theta)^2\le
 3\int_0^1M_c(\theta-v)^2W_U(v)\,dv.
\end{equation}
No Fourier support is asserted for $M_c$ itself.

For every $\eta\in\R$, orthogonality and Cauchy-Schwarz imply
\begin{equation}\label{eq:Fejer-all}
 \sum_{b\bmod r}K_m(\eta+b/r)
 =\frac rm\sum_{s\bmod r}
       \bigg|\sum_{\substack{0\le j<m\\j\equiv s\pmod r}}
                         \e(j\eta)\bigg|^2
 \le\frac rm\Big(\frac mr+1\Big)m=m+r.
\end{equation}
Thus $\sum_bW_U(\eta+b/r)\le5U+3r$.
After translating the integration variable in
\eqref{eq:positive-reproducing-majorant}, we find
\[
 \sum_{b\bmod r}M_c(\theta+b/r)^2
 \le3\int_0^1M_c(t)^2\sum_{b\bmod r}W_U(\theta+b/r-t)\,dt.
\]
Equation~\eqref{eq:Fejer-all} and Proposition~\ref{prop:CH}
prove \eqref{eq:maximal-all-residues}.

For reduced residues, first suppose $m\le r$.
Apply Lemma~\ref{lem:short-residue} to
$c'_n=m^{-1/2}\e((n-1)\eta)$, $1\le n\le m$.
This vector has squared norm $1$; shifting the exponents from
$0,\ldots,m-1$ to $1,\ldots,m$ only multiplies the polynomial by a
number of modulus one. It follows that
\[
 \frac1{\varphi(r)}\sum_{b\in\mathcal R(r)}K_m(\eta+b/r)
 \ll\log\log(3m).
\]
We need this for $m=U,2U$. If $m=2U>r$, use
\eqref{eq:Fejer-all} instead. Since $U\le r<2U$,
\[
 \frac1{\varphi(r)}\sum_{b\in\mathcal R(r)}K_{2U}(\eta+b/r)
 \le\frac{2U+r}{\varphi(r)}
 \le\frac{3r}{\varphi(r)}\ll\log\log(3U).
\]
We therefore have, in both cases,
\[
 \frac1{\varphi(r)}\sum_{b\in\mathcal R(r)}W_U(\eta+b/r)
 \ll\log\log(3U).
\]
Sum \eqref{eq:positive-reproducing-majorant} over these residues,
translate the integration variable as above, and apply
Proposition~\ref{prop:CH}. This proves
\eqref{eq:maximal-reduced-residues} with an effective constant.
\end{proof}

\begin{lemma}\label{lem:covering}
For $P>0$, $c\in\C^U$, $a\in\Z$, $r\in\N$, $(a,r)=1$, and
$\beta\ne0$,
\begin{equation}\label{eq:covering}
 \sum_{b\bmod r}\int_P^{2P}M_c(ab/r+\beta t)^2\,dt
 \ll (rP+|\beta|^{-1})E(c).
\end{equation}
The constant is absolute and effective.
\end{lemma}

\begin{proof}
Substitute $\theta=ab/r+\beta t$ in each integral. The $r$ image
intervals have length $\ell=|\beta|P$, and their translation parameters
are a permutation of $0,1/r,\ldots,(r-1)/r$ modulo one.
For a fixed point of the circle, its covering multiplicity, counting
multiple windings, is bounded by the number of points of a lattice of
spacing $1/r$ in an interval of length $\ell$. This is at most
$r\ell+2$. Reversing endpoints when $\beta<0$ gives the same bound.
Hence the left side of \eqref{eq:covering} is at most
\[
 \Big(rP+\frac2{|\beta|}\Big)\int_0^1M_c(\theta)^2\,d\theta.
\]
Proposition~\ref{prop:CH} completes the proof.
\end{proof}

\begin{lemma}\label{lem:variation}
Let $I$ be a compact interval of length $h>0$ and let $F:I\to\C$
be continuously differentiable. Then
\begin{equation}\label{eq:interval-variation}
 \sup_{t\in I}|F(t)|^2
 \le\frac2h\int_I|F(t)|^2\,dt
       +2h\int_I|F'(t)|^2\,dt.
\end{equation}
\end{lemma}

\begin{proof}
For $s,t\in I$, the identity $F(t)=F(s)+\int_s^tF'(v)\,dv$
and Cauchy-Schwarz give
$|F(t)|^2\le2|F(s)|^2+2h\int_I|F'(v)|^2\,dv$.
Average over $s\in I$, then take the supremum over $t$.
\end{proof}

\section{The mechanism: where the logarithms come from}\label{sec:mechanism}

This section explains the estimates used in the proof. It distinguishes
losses in the cited arguments from limitations of those particular
estimates; no necessity claim about all possible methods is intended.
The details of the present argument are proved in
Sections~\ref{sec:small}, \ref{sec:fourier},
\ref{sec:effective-sieve}, and~\ref{sec:proof}.

\subsection{The bilinear reduction}
Put $L=\log N$. Lemma~\ref{lem:log} reduces the problem to bounding
\begin{equation}\label{eq:bilinear-shape}
 \sum_{p\le N}f(p)\log p\sum_{n\le N/p}f(n)\e(\alpha pn)
\end{equation}
by $O_A(N)$ plus $L$ times the second term of \eqref{eq:main-bound}.
The hyperbolic truncation couples the two variables. Our argument uses
Cauchy-Schwarz and the coefficient bounds in \eqref{eq:class}, followed
by an upper-bound sieve for the resulting nonnegative expression.
This is a choice of method, not an assertion that multiplicativity
admits no other useful exploitation.

\subsection{The Montgomery-Vaughan argument}
In \cite[Sections~3-5]{MV}, the region $pn\le N$ is decomposed into
rectangles and a remainder. The refinement depths satisfy
$J_i\ll\log(2N/q)$; they need not all be comparable to this upper bound.
Cauchy-Schwarz in the integer variable leaves a square in the prime
variable. After triangular smoothing, its off-diagonal terms involve
pairs of primes with difference $h$. The prime-pair sieve bound
supplies two logarithmic savings, which offset the two prime
logarithmic weights in the upper bound. The diagonal terms and the
remainder of the partition are estimated separately.

For the auxiliary expression in \cite[equations~(18)-(19)]{MV}, put
$\|t\|=\min_{k\in\Z}|t-k|$. For the interval-length parameters $X,Y\ge1$ used there, with
$q\le XY$ and the other hypotheses of that fundamental estimate,
\[
 \sum_{1\le h\le X}\frac h{\varphi(h)}
       \min\Big(Y,\frac1{Y\|ha/q\|^2}\Big)
 \ll \frac{XY}{\varphi(q)}+X+Y\log(2X)+q\log(2XY/q).
\]
The minimum is interpreted as $Y$ when $ha/q$ is an integer.
The last logarithm arises in a harmonic divisor summation and enters
the bilinear bound under a square root. In the summation over the
refinement layers in \cite[Section~5]{MV}, it combines with the upper
bound for $J_i$ to give the displayed power $3/2$. The disjointness
hypotheses of \cite[equation~(16)]{MV} do not directly allow all layers
to be merged into one family. This observation does not exclude a
different treatment of their overlaps. The rational estimate is then
transferred to the $R$-form in \cite[Section~6]{MV}.

\subsection{Bachman's displacement logarithm}
Bachman introduces the completed weight in
\cite[equation~(2.7)]{Bachman}, using Siegel-Walfisz for its distribution
in the coprime residue classes. In the notation of this paper his
parameter $x/\bar x$ equals $B$. The subdivision in
\cite[equation~(2.15), p.~50]{Bachman} introduces weights $1/j$ for
$1\le j\le\lceil B\rceil$. In the estimate leading to
\cite[equation~(3.1), p.~51]{Bachman}, their sum is
$O(\log(2B))$. The square root yields the factor
$\sqrt{\log(2B)}$ in \eqref{eq:bachman-B}.

\subsection{Retaining the prefixes}
For a block $P<p\le2P$, set $U=\lfloor N/P\rfloor$ and
$c_n=f(n)$ for $n\le U$. We majorize the moving prefix by
$M_c(\alpha p)$, with $M_c$ as in \eqref{eq:maximal-definition}.
The Carleson-Hunt inequality controls this maximum in $L^2$ over a
period with an absolute constant, not a logarithmic loss. This avoids
the further rectangular refinement of the hyperbolic cutoff in the
remaining blocks; the dyadic subdivision in $P$ is still used.
Cauchy-Schwarz in the prime variable leaves the nonnegative expression
$Q_P^*(c)$ of \eqref{eq:Qstar-def}; this orientation is also Bachman's
\cite[equation~(2.6)]{Bachman}. Expanding a fixed prefix square here
would produce pairs of integer indices $n,m$, rather than prime pairs.
We instead retain the square and compare its values at primes with
integrals of maximal functions.

On each short interval, the Brun-Titchmarsh bound controls the prime
mass in every reduced residue class. Lemma~\ref{lem:variation} is
applied separately to the smooth prefix polynomials, after which the
bounds are majorized by $M_c$ and by $M_d$, where $d_n=nc_n$.
This proves the positive transfer estimate
\eqref{eq:positive-sieve-transfer}. Neither the maximum nor the moving
integer cutoff is differentiated. There is no prime-pair divisor
weight at this step and no termwise harmonic summation of the
oscillatory pair kernels.

\subsection{The interval length and the arithmetic factor}
Put $B=\max\{1,N|\beta|\}$ and $D=rB$. The sufficient condition
$P\ge16D^2$, with $h=P/\lceil B\rceil$, gives
\[
 h|\beta|U\le1,\qquad \frac h r\ge2\sqrt P.
\]
After applying the same sampling or covering estimate to $c$ and $d$,
the derivative contribution is controlled by
$h^2|\beta|^2E(d)\le h^2|\beta|^2U^2E(c)\le E(c)$.
The second inequality above makes $\log(2h/r)$ comparable to
$\log P$, so the sieve denominator offsets the weight $\log p$.
These are sufficient choices for this proof, not an optimality claim
for the local range $P\ge16D^2$. The case $\beta=0$ uses maximal
sampling, not a bound involving $|\beta|^{-1}$.

The factor $1/\varphi(r)$ in the displacement term enters through
\eqref{eq:effective-prime-mass} and remains after the covering estimate.
For $B>1$, this gives
$Q_P^*(c)\ll(rP+N/B)E(c)/\varphi(r)$.
The short reduced-residue estimate has a different role: it handles
$P>N/r$, whose total contribution is $O_A(N)$ before division by $L$.
The iterated logarithm in \eqref{eq:R-bound} comes from the totient
bound \eqref{eq:totient} in the transference argument.
Remark~\ref{rem:totient-obstruction} rules out an absolute $O(E(c))$
replacement in that particular unrestricted reduced-residue inequality;
it does not establish the optimal totient dependence for multiplicative
coefficients.

\subsection{The global ranges and effectivity}
First handle $r\ge L^3$ by Proposition~\ref{prop:MV}. In the remaining
range $r<L^3$, the hypotheses \eqref{eq:main-approx} give $D\le L^3$:
this uses $D=r$ when $B=1$ and $D=rN|\beta|\le L^3$ otherwise.
Lemma~\ref{lem:small} then removes primes up to a dyadic
$Y\asymp D^2$. The remaining blocks are summed using
\eqref{eq:block} for $P\le N/r$ and \eqref{eq:large-block} for $P>N/r$.
The choice $Q_0=N/L^3$, also used by Bachman
\cite[Theorems~1 and~4]{Bachman}, is convenient for these reductions;
no maximal range of validity is claimed for this choice.

The prime-block argument requires only the effective one-sided
Brun-Titchmarsh estimate, not an asymptotic in arithmetic progressions.
It therefore avoids the ineffective Siegel-Walfisz input in the
completion route. The remaining constants are accounted for in
Section~\ref{sec:effectivity}. This avoids the ineffective input; it
does not make Siegel's constant computable.

\section{Effective prime-block estimates}\label{sec:effective-sieve}

\begin{proposition}\label{prop:effective-maximal-prime}
Let $N>P$, $U=\lfloor N/P\rfloor$, and
$\alpha=a/r+\beta$, where $a\in\Z$, $r\in\N$, $(a,r)=1$, and
$\beta\in\R$. Set $B=\max\{1,N|\beta|\}$ and $D=rB$.
For $c\in\C^U$, write $E=E(c)$ and
\begin{equation}\label{eq:Qstar-def}
 Q_P^*(c)=\sum_{P<p\le2P}M_c(\alpha p)^2\log p.
\end{equation}
If $P\ge16D^2$, then
\begin{equation}\label{eq:effective-prime-general}
 Q_P^*(c)\ll\frac{rP+N/B}{\varphi(r)}E.
\end{equation}
If also $U\le r$, then
\begin{equation}\label{eq:effective-prime-short}
 Q_P^*(c)\ll P\log\log(3U)E.
\end{equation}
The constants are absolute and effective. There is no multiplicativity
hypothesis on $c$.
\end{proposition}

\begin{proof}
Put $J=\lceil B\rceil$ and $h=P/J$. Partition $(P,2P]$ into the
half-open intervals
$I_j=(P+(j-1)h,P+jh]$, $1\le j\le J$.
Their closures are used when taking suprema and integrals. Since
$B\le J\le2B$, $PU\le N$, and $N|\beta|\le B$,
\begin{equation}\label{eq:interval-parameters}
 h\ge\frac{P}{2B},\qquad |\beta|hU\le1,\qquad
 \frac h r\ge\frac{P}{2D}\ge2\sqrt P.
\end{equation}
Also $P\ge16D^2\ge16r^2>r$, so every prime in $(P,2P]$ is
coprime to $r$.

For $b\in\mathcal R(r)$, Proposition~\ref{prop:BT} applies to $I_j$
with upper endpoint $P+jh$ and length $h$. It gives
\begin{equation}\label{eq:effective-prime-mass}
 \sum_{\substack{p\in I_j\\p\equiv b\pmod r}}\log p
 \le C_\mathrm{BT}\frac{h\log(2P)}{\varphi(r)\log(2h/r)}
 \ll\frac h{\varphi(r)}.
\end{equation}
The last constant is effective: by \eqref{eq:interval-parameters},
$\log(2h/r)\ge\log(4\sqrt P)$, whose ratio to $\log(2P)$ is bounded
below by a positive absolute constant.

For $0\le K\le U$ put
\[
 F_{K,b}(t)=\sum_{n\le K}c_n\e\bigl(n(ab/r+\beta t)\bigr),
 \qquad d_n=nc_n\quad(1\le n\le U).
\]
Then $E(d)\le U^2E$ and
$|F'_{K,b}(t)|\le2\pi|\beta|M_d(ab/r+\beta t)$.
Apply Lemma~\ref{lem:variation} to each fixed $K$ on the closure of
$I_j$. Bounding its two integrands by the corresponding maximal
expressions gives a bound independent of $K$, and hence
\[
 \sup_{t\in I_j}M_c(ab/r+\beta t)^2
 \le\frac2h\int_{I_j}M_c(ab/r+\beta t)^2\,dt
 +8\pi^2h|\beta|^2\int_{I_j}M_d(ab/r+\beta t)^2\,dt.
\]
For a prime $p\equiv b\pmod r$, periodicity gives
$M_c(\alpha p)=M_c(ab/r+\beta p)$.
Multiply the last inequality by the prime-mass bound
\eqref{eq:effective-prime-mass}, and sum over $j$ and $b$.
We obtain the positive transfer estimate
\begin{equation}\label{eq:positive-sieve-transfer}
 Q_P^*(c)\ll\frac1{\varphi(r)}\sum_{b\in\mathcal R(r)}
 \int_P^{2P}\big(M_c(ab/r+\beta t)^2
   +h^2|\beta|^2 M_d(ab/r+\beta t)^2\big)\,dt.
\end{equation}
Only the fixed prefix polynomials were differentiated, not the maximum
and not any prime-dependent truncation.

If $B>1$, enlarge the residue sum to all classes and apply
Lemma~\ref{lem:covering}, once to $c$ and once to $d$.
Equation~\eqref{eq:positive-sieve-transfer} is then
\[
 \ll\frac{rP+|\beta|^{-1}}{\varphi(r)}
                  (E+h^2|\beta|^2E(d))
 \le\frac{2(rP+N/B)}{\varphi(r)}E
\]
up to an effective absolute constant. Here
$|\beta|^{-1}=N/B$ and \eqref{eq:interval-parameters} was used.
This proves \eqref{eq:effective-prime-general} in this case.

If $B=1$, including $\beta=0$, apply
\eqref{eq:maximal-all-residues} at each fixed $t$, again to both vectors.
It bounds \eqref{eq:positive-sieve-transfer} by
\[
 \ll\frac{P(U+r)}{\varphi(r)}
                  (E+h^2|\beta|^2E(d))
 \ll\frac{N+rP}{\varphi(r)}E.
\]
This is \eqref{eq:effective-prime-general} with $B=1$.

Finally, when $U\le r$, keep reduced residues in
\eqref{eq:positive-sieve-transfer} and apply
\eqref{eq:maximal-reduced-residues} at each $t$.
The outcome is
\[
 Q_P^*(c)\ll P\log\log(3U)
                  (E+h^2|\beta|^2E(d))
 \ll P\log\log(3U)E.
\]
This proves \eqref{eq:effective-prime-short} for every $\beta$, including zero.
\end{proof}

\begin{proposition}\label{prop:arbitrary-prefixes}
Use the parameter hypotheses of Proposition~\ref{prop:effective-maximal-prime},
including $P\ge16D^2$. Let $a_p$ for $P<p\le2P$ and $c_n$ for
$1\le n\le U$ be complex coefficients satisfying
\[
 \sum_{P<p\le2P}|a_p|^2\log p\le A_1^2P,\qquad
 \sum_{n\le U}|c_n|^2\le A_2^2U,
\]
where $A_1,A_2\ge0$. For any integers $K(p)$ with $0\le K(p)\le U$, put
\[
 \mathcal T_{a,c,K}(P)=
 \sum_{P<p\le2P}a_p\log p\sum_{n\le K(p)}c_n\e(\alpha pn).
\]
Then
\begin{equation}\label{eq:arbitrary-block}
 |\mathcal T_{a,c,K}(P)|\ll A_1A_2
 \bigg(\frac{N}{\sqrt{\varphi(r)B}}
 +\sqrt{\frac{rNP}{\varphi(r)}}\bigg).
\end{equation}
If $U\le r$, one also has
\begin{equation}\label{eq:arbitrary-short-block}
 |\mathcal T_{a,c,K}(P)|
 \ll A_1A_2N\sqrt{\frac{\log\log(3U)}U}.
\end{equation}
The constants are absolute and effective, independently of the choices
of $K(p)$ and of the coefficients.
\end{proposition}

\begin{proof}
Cauchy-Schwarz gives
\[
 |\mathcal T_{a,c,K}(P)|^2\le A_1^2P Q_P^*(c).
\]
By \eqref{eq:effective-prime-general} and $PU\le N$, this is
\[
 \ll A_1^2A_2^2\frac{P(rP+N/B)U}{\varphi(r)}
 \le A_1^2A_2^2
       \Big(\frac{rNP}{\varphi(r)}+\frac{N^2}{\varphi(r)B}\Big)
\]
up to an absolute constant. Taking square roots proves
\eqref{eq:arbitrary-block}.
Using \eqref{eq:effective-prime-short} instead gives
$|\mathcal T_{a,c,K}(P)|\ll A_1A_2P\sqrt{U\log\log(3U)}$.
The inequality $PU\le N$ proves \eqref{eq:arbitrary-short-block}.
\end{proof}

For $f\in\FA$, define
\begin{equation}\label{eq:T-def}
 T(P)=\sum_{P<p\le2P}f(p)\log p
                \sum_{n\le N/p}f(n)\e(\alpha pn).
\end{equation}
Taking $a_p=f(p)$, $c_n=f(n)$, and $K(p)=\lfloor N/p\rfloor$,
we have $0\le K(p)\le U$ and explicitly
\[
 E=\sum_{n\le U}|f(n)|^2\le A^2U,
 \qquad
 \sum_{P<p\le2P}|f(p)|^2\log p\ll A^2P.
\]
The second estimate follows from \eqref{eq:prime-elementary}.
Thus Proposition~\ref{prop:arbitrary-prefixes} gives
\begin{equation}\label{eq:block}
 |T(P)|\ll_A\frac{N}{\sqrt{\varphi(r)B}}
                         +\sqrt{\frac{rNP}{\varphi(r)}}
 \qquad(P\ge16D^2),
\end{equation}
and, when $U\le r$,
\begin{equation}\label{eq:large-block}
 |T(P)|\ll_A N\sqrt{\frac{\log\log(3U)}U}.
\end{equation}
For $p>N$ the chosen prefix is zero, so the final dyadic block needs no
special endpoint convention. These local estimates do not require
multiplicativity. The global reductions do; in particular, the local
coefficient generality does not by itself extend the final theorem
beyond $\FA$.

\section{Proofs of the main results}\label{sec:proof}

\begin{proof}[Proof of Theorem~\ref{thm:main}]
We may assume $N$ exceeds a sufficiently large effective absolute
constant. On the remaining bounded interval,
$|S_f(N,\alpha)|\le AN$ proves the assertion after an effective
increase in a constant depending only on $A$.
If $r\ge L^3$, then \eqref{eq:main-approx} gives
$L^3\le r\le N/L^3$ and $|\beta|\le r^{-2}$.
Proposition~\ref{prop:MV}, with $R_0=L^3$, gives
\[
 |S_f(N,\alpha)|\ll_A
 \frac NL+\frac{N(\log(2L^3))^{3/2}}{L^{3/2}}\ll_A\frac NL.
\]
The last inequality holds effectively for large $L$, since
$(\log L)^3/L\to0$.

Assume henceforth that $r<L^3$. Put $D=rB$ and choose $Y$ as in
\eqref{eq:polynomial-cutoff}. If $B=1$, then $D=r<L^3$; otherwise
$D=rN|\beta|\le L^3$ by \eqref{eq:main-approx}.
For all sufficiently large $N$ we have uniformly
\begin{equation}\label{eq:uniform-parameters}
 Y<\sqrt N,\qquad N/Y\ge16D^2,\qquad
 Y<N/r,\qquad r^2<N.
\end{equation}
The first two inequalities were checked in Lemma~\ref{lem:small}.
For the third, use $rY<32rD^2\le32L^9<N$; the fourth follows from
$r^2<L^6<N$. All these thresholds are effective and absolute.

Lemmas~\ref{lem:log} and~\ref{lem:small} now imply
\begin{equation}\label{eq:dyadic-reduction}
 L S_f(N,\alpha)=
 \sum_{\substack{Y\le P<N\\P\text{ dyadic}}}T(P)+O_A(N).
\end{equation}
Every prime $Y<p\le N$ belongs to exactly one displayed block.
Each block satisfies $P\ge Y\ge16D^2$, which verifies the sufficient
hypothesis of \eqref{eq:block} and \eqref{eq:large-block}.

For $Y\le P\le N/r$, use \eqref{eq:block}.
There are $O(L)$ such blocks, and hence the sum of their first terms is
\begin{equation}\label{eq:main-block-sum}
 \ll_A\frac{NL}{\sqrt{\varphi(r)B}}.
\end{equation}
The second terms form a geometric sum:
\begin{equation}\label{eq:geometric-sum}
 \sum_{\substack{Y\le P\le N/r\\P\text{ dyadic}}}
          \sqrt{\frac{rNP}{\varphi(r)}}
 \ll\sqrt{\frac{rN}{\varphi(r)}}\sqrt{\frac Nr}
 =\frac N{\sqrt{\varphi(r)}}\le N.
\end{equation}
If $N/r\ge N$ (that is, $r=1$), only the already specified blocks
$P<N$ are included; the same upper bound holds.

For the remaining blocks $N/r<P<N$, we have $U=\lfloor N/P\rfloor<r$,
so \eqref{eq:large-block} applies. Let $P_*$ be the largest power of
two strictly less than $N$ and put $\lambda=N/P_*\in(1,2]$.
For $P=P_*2^{-j}$, $j\ge0$,
\[
 2^j\le U=\lfloor\lambda2^j\rfloor\le2^{j+1}.
\]
Consequently the sum over the selected blocks is at most
\begin{equation}\label{eq:short-sum}
 \ll_A N\sum_{j\ge0}2^{-j/2}
           \sqrt{\log\log(3\cdot2^{j+1})}
 \ll_A N\sum_{j\ge0}2^{-j/2}\sqrt{\log(j+3)}
 \ll_A N.
\end{equation}
The convergent series has an effective absolute bound.
Combining \eqref{eq:dyadic-reduction}-\eqref{eq:short-sum}, we obtain
\[
 L|S_f(N,\alpha)|\ll_A
 N+\frac{NL}{\sqrt{\varphi(r)B}}.
\]
Division by $L$ proves \eqref{eq:main-bound}.
\end{proof}

Corollary~\ref{cor:R} is stated for $R\ge3$ to match
\eqref{eq:bachman-R}; for $2\le R<3$, the right-hand side of
\eqref{eq:R-bound} is $\gg N$, so \eqref{eq:l1-l2} gives the same
conclusion for $R\ge2$, as in \cite[Corollary~1]{MV}.

\begin{proof}[Proof of Corollary~\ref{cor:R}]
For $N$ in a fixed bounded interval the trivial estimate suffices.
For all larger $N$, $Q_0=N/(\log N)^3\ge1$.
Choose $a/r$ by Lemma~\ref{lem:dirichlet} with parameter $Q_0$, and put
$\beta=\alpha-a/r$, $B=\max\{1,N|\beta|\}$.
All hypotheses of Theorem~\ref{thm:main} hold.
We show that
\begin{equation}\label{eq:transfer-goal}
 \frac1{\varphi(r)B}\ll\frac{\log\log(3R)}R.
\end{equation}
If $R/2\le r<R$, use $B\ge1$, \eqref{eq:totient}, and
$1/r\le2/R$ to obtain \eqref{eq:transfer-goal}.
If $r\ge R$, the same conclusion follows since
$u\mapsto\log\log(3u)/u$ is decreasing for $u\ge3$.
Indeed its derivative is
\[
 \frac{1/\log(3u)-\log\log(3u)}{u^2}<0\qquad(u\ge3).
\]
No monotonicity property of $\varphi$ is required.

Suppose now $r<R/2$. Since $q\ge R$, the reduced fractions $a/r$ and
$a_0/q$ are distinct and $r<q/2$. Therefore
\[
 |\beta|\ge\frac1{rq}-\frac1{q^2}\ge\frac1{2rq}.
\]
It follows that
\[
 rB\ge rN|\beta|\ge\frac N{2q}\ge\frac R2,
\]
and hence, by \eqref{eq:totient} and $r<R$,
\[
 \frac1{\varphi(r)B}
 =\frac{r/\varphi(r)}{rB}
 \ll\frac{\log\log(3r)}R
 \le\frac{\log\log(3R)}R.
\]
This proves \eqref{eq:transfer-goal} in all cases.
Insert it into \eqref{eq:main-bound} to obtain \eqref{eq:R-bound}.
\end{proof}

\subsection{Effective dependence of the constants}\label{sec:effectivity}

The arithmetic input in the prime-block estimate is the one-sided,
effective inequality \eqref{eq:BT}. Its constant is uniform in the
interval location, its length, and the reduced residue class.
The constants in the logarithmic identity and in the prime-power
series are bounded by explicit convergent numerical series and powers
of $A$. The coarse near-rational estimate and the range $r\ge L^3$
use only the effective Montgomery-Vaughan bound
\eqref{eq:MV-input}. The finite Fourier lemmas depend on the effective
absolute Carleson-Hunt constant and the effective totient bound
\eqref{eq:totient}; their remaining steps are finite identities,
Cauchy-Schwarz, and integration. The elementary thresholds in
\eqref{eq:uniform-parameters} can also be found effectively.

It follows by composition that a constant $C(A)$ in each main result
can be bounded by a computable function of $A$. We have not optimized
or evaluated it numerically. The proofs of Theorem~\ref{thm:main}
and Corollary~\ref{cor:R} neither exclude an exceptional real
character nor use a zero-free region for Dirichlet $L$-functions.
The discussion of Siegel-Walfisz concerns the alternative completion
route, not a dependency of these proofs. The effective prime-mass
argument avoids the progression asymptotic used in the earlier
completion proof; it does not make Siegel's constant computable.

\subsection{A simultaneous exceptional set}\label{sec:exceptional}

Let $\mathbb T=\R/\Z$, with Lebesgue measure $\operatorname{meas}$
normalized by $\operatorname{meas}(\mathbb T)=1$. For $A\ge1$,
$N\ge3$, and $\lambda>0$ put
\begin{equation}\label{eq:exceptional-definition}
 \mathcal E_A(N,\lambda)
 =\{\alpha\in\mathbb T:\ |S_f(N,\alpha)|>\lambda N\ \text{for some}\ f\in\FA\},
\end{equation}
an open set, being a union of open superlevel sets of continuous
functions. Montgomery and Vaughan's Corollary~2 \cite{MV} gives
$|S_f(N,\alpha)|\ll_AN/\log N$ for almost all fixed $\alpha$ and all
$N>N_0(\alpha)$, uniformly in $f\in\FA$. The statement below fixes $N$
instead and bounds the measure of the set of phases that has to be
excluded for the whole class at once. It requires a count of
denominators with bounded totient; the count is classical, and the
elementary proof given keeps the constant effective.

\begin{lemma}\label{lem:inverse-totient}
There is an effective absolute constant $C_\varphi$ such that
\begin{equation}\label{eq:inverse-totient-count}
 \#\{r\in\N:\ \varphi(r)\le H\}\le C_\varphi H\qquad(H\ge1).
\end{equation}
\end{lemma}

\begin{proof}
Let $g$ be the multiplicative function with
$g(p)=(1-1/p)^{-2}-1=(2p-1)/(p-1)^2$ and $g(p^j)=0$ for $j\ge2$. Then
$(n/\varphi(n))^2=\sum_{d\mid n}g(d)$, both sides being multiplicative
with value $(1-1/p)^{-2}$ at every prime power $p^k$, and
$0\le g(p)/p\le6/p^2$, with equality at $p=2$. Hence, for real
$X\ge1$,
\[
 \sum_{n\le X}\Big(\frac n{\varphi(n)}\Big)^2
 =\sum_{d\le X}g(d)\Big\lfloor\frac Xd\Big\rfloor\le C_2X,
 \qquad
 C_2=\prod_p\Big(1+\frac{g(p)}p\Big)\le\exp\Big(6\sum_{m\ge2}\frac1{m^2}\Big)\le e^6.
\]
There are at most $H$ integers $r\le H$, and for $j\ge0$ an integer
$2^jH<r\le2^{j+1}H$ with $\varphi(r)\le H$ has $(r/\varphi(r))^2>4^j$,
so at most $4^{-j}C_22^{j+1}H=2^{1-j}C_2H$ such $r$ exist. Summing
over $j$ gives \eqref{eq:inverse-totient-count} with
$C_\varphi=1+4C_2$.
\end{proof}

\begin{corollary}\label{cor:simultaneous}
Let $A\ge1$. There are effective constants $c_A\ge6$ and $K_A>0$,
depending only on $A$, such that for every real $N\ge3$ and every
$\lambda$ with $c_A/\log N\le\lambda\le1$,
\begin{equation}\label{eq:exceptional-bound}
 \operatorname{meas}\mathcal E_A(N,\lambda)\le\frac{K_A}{N\lambda^4}.
\end{equation}
In particular, for $N\ge e^{c_A}$ there is a set
$\mathcal E_A(N)\subseteq\mathbb T$ of measure $\ll_A(\log N)^4/N$
such that $|S_f(N,\alpha)|\le c_AN/\log N$ for all
$\alpha\notin\mathcal E_A(N)$ and all $f\in\FA$.
\end{corollary}

\begin{proof}
Let $C=C(A)\ge1$ be an admissible constant in \eqref{eq:main-bound},
put $c_A=\max\{6,2C\}$, and let $c_A/L\le\lambda\le1$, where
$L=\log N$; then $L\ge6$ and $Q_0=e^L/L^3\ge1$. Let
$\alpha\in\mathcal E_A(N,\lambda)$, with a witness $f$, and choose
$a/r$ by Lemma~\ref{lem:dirichlet} with parameter $Q_0$, so that
\eqref{eq:main-approx} holds with $\beta=\alpha-a/r$. Since
$CN/L\le\lambda N/2$, Theorem~\ref{thm:main} forces
$CN/\sqrt{\varphi(r)B}>\lambda N/2$, that is,
\[
 \varphi(r)B<H:=\Big(\frac{2C}\lambda\Big)^2,
 \qquad\text{so that}\qquad
 \varphi(r)\le H,\qquad|\beta|\le\frac BN<\frac H{N\varphi(r)}.
\]
Reducing $a$ modulo $r$, we see that $\mathcal E_A(N,\lambda)$ is
covered by the arcs of radius $H/(N\varphi(r))$ about the points
$a/r$ with $\varphi(r)\le H$ and $a\in\mathcal R(r)$. There are
$\varphi(r)$ arcs for each $r$, each of measure at most
$2H/(N\varphi(r))$, and Lemma~\ref{lem:inverse-totient} gives
\[
 \operatorname{meas}\mathcal E_A(N,\lambda)
 \le\sum_{\varphi(r)\le H}\frac{2H}N\le\frac{2C_\varphi H^2}N
 =\frac{2C_\varphi(2C)^4}{N\lambda^4}.
\]
This is \eqref{eq:exceptional-bound} with $K_A=2C_\varphi(2C)^4$. For
$N\ge e^{c_A}$ take $\lambda=c_A/L$ and
$\mathcal E_A(N)=\mathcal E_A(N,c_A/L)$.
\end{proof}

For a fixed $f$, Parseval's identity gives an exceptional set of
measure at most $A^2/(N\lambda^2)$, but that set may depend on $f$;
the corollary controls the union over the class. For pointwise
$1$-bounded multiplicative functions, the same measure order follows
from \eqref{eq:bachman-B}, not necessarily with effective constants.
Indeed, for $c/\log N\le\lambda\le1$ with a sufficiently large
absolute $c$, a value exceeding $\lambda N$ forces
$\varphi(r)B/\log(2B)<H_{\mathrm B}$ with
$H_{\mathrm B}\asymp\lambda^{-2}$.
In the range $2^j\le B<2^{j+1}$, put
$Y_j=H_{\mathrm B}(j+2)2^{-j}$. The possible denominators have
$\varphi(r)\le Y_j$, and the approximation error is at most
$2^{j+1}/N$. Lemma~\ref{lem:inverse-totient} gives
$\sum_{\varphi(r)\le Y_j}\varphi(r)\ll Y_j^2$ when $Y_j\ge1$;
there are no such denominators otherwise. The covering arcs therefore
have total measure
\[
 \ll\frac{H_{\mathrm B}^2}{N}
       \sum_{j\ge0}(j+2)^2 2^{-j}
 \ll\frac1{N\lambda^4}.
\]
This is a deduction from Bachman's bound, not a formulation quoted
from his paper. The present corollary supplies the full class $\FA$
and effective constants; no optimality of the exponent is asserted.

\section{The displacement exponent and a Fourier obstruction}\label{sec:sharpness}

\begin{proposition}\label{prop:sharpness}
For positive integers $N\to\infty$, put $B=\log N$, $\alpha=B/N$, and
$f_N(n)=n^{-i\pi B}$. Then $f_N\in\mathcal F_1$, the auxiliary
approximation $0/1$ satisfies \eqref{eq:main-approx}, and
\begin{equation}\label{eq:sharpness}
 |S_{f_N}(N,\alpha)|\sim\frac{N}{\sqrt{2B}}.
\end{equation}
In particular, for any fixed $\eps>0$, replacing $B^{-1/2}$
by $B^{-1/2-\eps}$ in \eqref{eq:main-bound} is false uniformly
in the same class and range, even when $r=1$.
\end{proposition}

\begin{proof}
The function is completely multiplicative and has modulus one.
The approximation $0/1$ has displacement $B/N\le(\log N)^3/N$
for large $N$, and $1\le Q_0$ then. Thus the parameter called $B$ in
the theorem is exactly the present $B$.

Let
\[
 G(x)=\exp\bigl(i(2\pi Bx/N-\pi B\log x)\bigr).
\]
Comparison on each unit interval gives
\[
 \bigg|\sum_{n=1}^NG(n)-\int_1^NG(x)\,dx\bigg|
 \le1+\int_1^N|G'(x)|\,dx\ll1+B\log N.
\]
After $x=Nu$, and extension of the lower endpoint from $1/N$ to $0$
at a cost at most $1$ after multiplication by $N$, it remains to
estimate
\[
 I(B)=\int_0^1 e^{iB\psi(u)}\,du,
 \qquad \psi(u)=2\pi u-\pi\log u.
\]
This improper integral is absolutely convergent, since the integrand
has modulus one. Its only stationary point is $u_0=1/2$, with $\psi''(u_0)=4\pi$.
We give the local stationary-phase calculation to specify the error.

Choose a fixed smooth function $\chi$ supported in a sufficiently
small interval about $u_0$, equal to $1$ near $u_0$.
For the part with amplitude $1-\chi$, integration by parts against
$(iB\psi'(u))^{-1}(e^{iB\psi(u)})'$ gives $O(B^{-1})$.
This is valid at $0$: the function
$1/\psi'(u)=u/(\pi(2u-1))$ is $O(u)$ there and its derivative is
bounded near $0$. Away from $u_0$ the remaining derivative and
boundary terms are integrable and bounded.

On the support of $\chi$, make the smooth change of variables
\[
 y=\operatorname{sgn}(u-u_0)
       \sqrt{\frac{\psi(u)-\psi(u_0)}{2\pi}}.
\]
Taylor's formula gives $y'(u_0)=1$, so this is a local diffeomorphism.
If $u=u(y)$ is its inverse, the smooth compactly supported amplitude
$a(y)=\chi(u(y))u'(y)$ satisfies $a(0)=1$, and
\[
 \int_0^1\chi(u)e^{iB\psi(u)}\,du
 =e^{iB\psi(u_0)}\int_{\R}a(y)e^{2\pi iBy^2}\,dy.
\]
Choose a fixed smooth compactly supported $\eta$ equal to $1$ near
$0$. There is a smooth compactly supported $b$ with
$a(y)=\eta(y)+yb(y)$. Integrating the second term using
$(e^{2\pi iBy^2})'=4\pi iBy e^{2\pi iBy^2}$ bounds it by
$(4\pi B)^{-1}\int_{\R}|b'(y)|\,dy$. Integrating by parts on the two tails where
$1-\eta$ is nonzero shows that
\[
 \int_{\R}\eta(y)e^{2\pi iBy^2}\,dy
 =\int_{\R}e^{2\pi iBy^2}\,dy+O(B^{-1})
 =\frac{e^{i\pi/4}}{\sqrt{2B}}+O(B^{-1}).
\]
For the last equality, insert $e^{-\delta y^2}$, evaluate the Gaussian
integral as $\sqrt\pi(\delta-2\pi iB)^{-1/2}$ with positive-real-part
branch, and let $\delta\downarrow0$. Integration by parts bounds the
tails beyond $|y|=T$ by $O((BT)^{-1})$, uniformly in $\delta\ge0$,
so the limit agrees with the improper Fresnel integral.
We have proved
\[
 I(B)=\frac{e^{iB\psi(1/2)+i\pi/4}}{\sqrt{2B}}+O(B^{-1}).
\]
Therefore
\begin{equation}\label{eq:sharpness-expansion}
 S_{f_N}(N,\alpha)
 =\frac{N}{\sqrt{2B}}
   \exp\bigl(iB\psi(1/2)-i\pi B\log N+i\pi/4\bigr)
 +O(N/B+B\log N+1).
\end{equation}
When $B=\log N$, the error is $o(N/\sqrt B)$, which gives
\eqref{eq:sharpness}.
Finally,
$N/\log N+N/B^{1/2+\eps}=o(N/\sqrt B)$ for every fixed
$\eps>0$. The dependence of $f_N$ on $N$ is permitted in a
bound uniform over the coefficient class, so this disproves the proposed
stronger exponent.
\end{proof}

\begin{remark}\label{rem:totient-obstruction}
The remaining totient factor is a different issue. In
Lemma~\ref{lem:short-residue}, take $U=r$, $b_0\in\mathcal R(r)$, and
$c_n=\e(-b_0n/r)$ for $1\le n\le r$. Then $E(c)=r$ and exact
orthogonality gives
\[
 \frac1{\varphi(r)}\sum_{b\in\mathcal R(r)}
       \bigg|\sum_{n\le r}c_n\e(bn/r)\bigg|^2
 =\frac{r^2}{\varphi(r)}=\frac r{\varphi(r)}E(c).
\]
Since $r/\varphi(r)$ is unbounded, this disproves an absolute
$O(E(c))$ bound for this particular reduced-residue mean over
arbitrary vectors. It does not prove
optimality of the totient factor for multiplicative functions,
since the displayed vector need not be multiplicative. Improving
the final arithmetic dependence would require additional
information or a change in the reduction to that unrestricted mean.
\end{remark}

\section{Conclusion and future work}\label{sec:conclusion}

We have refined the Montgomery-Vaughan bound, and Bachman's sharpening
of it, in the original coefficient class, removing a logarithmic
factor with effective implied constants. The argument combines short-interval
upper bounds for primes with maximal Fourier estimates, and the
square-root dependence on the displacement parameter is sharp in the
stated uniform setting.

The remaining arithmetic dependence presents a different question.
Remark~\ref{rem:totient-obstruction} rules out replacing the
right-hand side of \eqref{eq:short-residue} by $O(E(c))$ uniformly
for arbitrary coefficient vectors. This does not establish the
optimal arithmetic factor in the multiplicative-function theorem.
A possible direction is to retain multiplicativity or other
structure before passing to the unrestricted coefficient estimate,
or to change that reduction; this will be the subject of future
research. No improvement of the totient factor is proved here.

A second direction concerns the range of $D=rB$. The hypothesis
$P\ge16D^2$ of Proposition~\ref{prop:effective-maximal-prime} makes
the Brun-Titchmarsh bound lossless against the weight $\log p$; below
the Brun-Titchmarsh scale the same argument still gives a bound, at
the cost of a logarithm. Let $B>1$ and $r\le P<r\lceil B\rceil$, and
let $J=\lceil B\rceil$, $h=P/J$ and the intervals $I_j$ be as in the
proof of Proposition~\ref{prop:effective-maximal-prime}, so that
$h<r$. Every prime in $(P,2P]$ exceeds $r$ and is therefore coprime
to $r$, and an interval of length $h<r$ contains at most one integer
in each class modulo $r$; hence
$\sum_{p\in I_j,\ p\equiv b\ (\mathrm{mod}\ r)}\log p\le\log(2P)$ for
every $b\in\mathcal R(r)$ and every $j$. Using this in place of
\eqref{eq:effective-prime-mass}, the argument leading to
\eqref{eq:positive-sieve-transfer} gives
\[
 Q_P^*(c)\ll\log(2P)\sum_{b\in\mathcal R(r)}\int_P^{2P}
 \Big(\frac1hM_c(ab/r+\beta t)^2+h|\beta|^2M_d(ab/r+\beta t)^2\Big)\,dt.
\]
Enlarging the sum to all $b$ modulo $r$ and applying
Lemma~\ref{lem:covering} to both integrals, with $E(d)\le U^2E(c)$ and
$h|\beta|U\le B/J\le1$, we obtain
\[
 Q_P^*(c)\ll\log(2P)\Big(rP+\frac NB\Big)\Big(\frac1h+h|\beta|^2U^2\Big)E(c)
 \ll\log(2P)\Big(D+\frac NP\Big)E(c),
\]
since $rP/h=rJ\le2D$, $(N/B)/h=NJ/(BP)\le2N/P$,
$rP\,h|\beta|^2U^2\le rP|\beta|U\le D$, and
$(N/B)\,h|\beta|^2U^2\le(N/B)|\beta|U\le N/P$. The proof of
Proposition~\ref{prop:arbitrary-prefixes} then gives, in this range,
\[
 |T(P)|\ll_A\sqrt{ND\log(2P)}+N\sqrt{\frac{\log(2P)}P}.
\]
For $P\ge r\lceil B\rceil$ one may instead keep the Brun-Titchmarsh
factor $\log(2P)/\log(2P/(rJ))$ arising from \eqref{eq:BT} with
$h=P/J$. Together with the small-prime range of Lemma~\ref{lem:small},
these bounds suggest a three-range argument for larger $D$; we will
pursue this optimization in future work.

\section*{Acknowledgement}
The author is grateful to Dirk Zeindler for feedback on the manuscript
and for fruitful discussions. The original idea and conception of the
paper, together with the overall strategy of the proof and the first
draft, are due to the author.

Subsequently, ChatGPT Sol 5.6 was used for exposition and final revision. The author led the effort to make the implied constants effective, with its assistance. An adversarial audit of an earlier proof of the main theorem led to the effective form of the final estimates. The revised proof retains the maximal Fourier approach but replaces the Siegel-Walfisz completion step by a one-sided Brun-Titchmarsh argument on short intervals, combined with maximal sampling and an elementary variation estimate; see Section~\ref{sec:effective-sieve} and \eqref{eq:positive-sieve-transfer}. This replacement avoids the ineffective input rather than making the constant in Siegel's theorem computable; see Section~\ref{sec:effectivity}. All mathematical content was verified by the author, who takes full responsibility for it.

\end{document}